\documentclass[12pt]{amsart}

\usepackage{amsfonts}
\usepackage{amscd}
\usepackage{amssymb}
\usepackage{enumerate}
\allowdisplaybreaks

\usepackage{color}

\newtheorem{thm}{Theorem}[section]

\theoremstyle{definition}

\theoremstyle{remark}
\newtheorem{rem}[thm]{Remark}
\numberwithin{equation}{section}

\newcommand{\R}{\mathbb R}

\newcommand{\C}{{\mathbb C}}

\newcommand{\be}{\begin{equation}}
\newcommand{\ee}{\end{equation}}

\renewcommand{\H}{\mathbb H}
\title[ Fractional powers of the Hermite operator]
{A note on fractional powers of \\ the Hermite operator}

\author[ S. Thangavelu]{Sundaram Thangavelu}

\address[S. Thangavelu]{Department of Mathematics\\
 Indian Institute of Science\\
560 012 Bangalore, India}

\email{veluma@math.iisc.ernet.in}

\thanks{This work is  supported by the J. C. Bose Fellowship of
the author from the Department of Science and Technology,
Government of India.}
\begin{document}

\maketitle
\begin{abstract} We give a very short proof of  a result proved by Cappiello-Rodino-Toft on the Weyl symbol of the inverse of the Harmonic oscillator. We also extend their results to  fractional powers.

\end{abstract}

\section{Introduction} 
In  2015 Cappiello, Rodino and Toft \cite{CRT} have considered the inverse of the Hermite operator $ H = -\Delta+|x|^2 $ on $ \R^n $ as a Weyl pseudodifferential operator and proved certain estimates of Gevrey type for the symbol. They have also obtained an explicit expression for the symbol when the dimension is even. The aim of this note is to give simple proofs of their results making use of the connection between the Weyl tranform (which is related to the Schr\"odinger representation on the Heisenberg group $ \H^n $) and the  Hermite semigroup. 

As is well known the spectrum of the operator $ H $ consists of $ (2k+n), k \in \mathbb{N} $ and hence it is invertible. The formal inverse can be written in terms of the spectral theorem by
$$ H^{-1}  = \sum_{k=0}^\infty (2k+n)^{-1} P_k $$ where $ P_k $ are the orthogonal projections associated to the eigenspaces corresponding to the eigenvalues $ (2k+n).$ However, it is known that $ H^{-1} $ is a pseudo-differential operator with a symbol $ b(x,\xi) $ in the Weyl calculus. Thus
$$  H^{-1}\varphi(\xi) = (2\pi)^{-n} \int_{\R^n} \int_{\R^n} e^{i(\xi-\eta)\cdot y}b(\frac{\xi+\eta}{2},y) \varphi(\eta) dy d\eta $$ for $ \varphi \in L^2(\R^n).$ In \cite{CRT} the authors have obtained the following explicit expression for the symbol $ b(x,\xi) $ when the dimension $ n $ is even.

\begin{thm}(Cappiello-Rodino-Toft)
Let $ b_{2n}(x,\xi) $ stand for the Weyl symbol of $ H^{-1} $ on $\R^{2n}.$ Then one has the explicit formula
$$ b_{2n}(x,\xi) =\sum_{j=0}^{n-1} \frac{(n+j-1)!}{(n-1)!j!} (-1)^j (2j)! \frac{1-p_{2j}(|x|^2+|\xi|^2) e^{-(|x|^2+|\xi|^2)}}{(|x|^2+|\xi|^2)^{2j+1}}$$
where $ p_j(t) $ are the Taylor polynomials of the function $ e^{-t} $ about $ t =0.$
\end{thm}
The proof given in \cite{CRT} is quite long and based on the fact that the symbol $ b $ satisfies a partial differential equation. In this note, the above theorem becomes an easy consequence of an integral representation for the symbol $ b $ which is based on the formula
$$ H^{-1}  = \int_0^\infty e^{-tH} dt $$
and the fact that $ e^{-tH} $ is a pseudodifferential operator with an explicit symbol. In the same paper \cite{CRT} the authors have proved the following result giving estimates on the derivatives of the symbol $ b$ of $ H^{-1}.$

\begin{thm} The following estimates on the Weyl symbol $ b(x,\xi) $ of the operator $ H^{-1}$ are valid: there exists a constant $ C > 0 $ such that for any $ \alpha \in \mathbb{N}^{2n} $ and  $ r \in [0,1]$
$$   |\partial_{x,\xi}^\alpha b(x,\xi) | \leq C^{|\alpha|+1} (\alpha!)^{(r+1)/2} (|x|^2+|\xi|^2)^{-1-(r/2)|\alpha|} .$$ 

\end{thm}

In this note we give a short proof of the above theorem. Actually we can consider $ H^{-s} $ for any $ s > 0 $ and prove similar estimates for the Weyl symbol $ b_s $ of the operator $ H^{-s}.$ We will also say something about conformally invariant fractional powers $ H_{-s} $ studied in the literature.

\section{Fractional powers $ H^{-s} $ of the Hermite operator}

In this section we consider fractional powers of the Hermite operator $ H = -\Delta+|x|^2 $ on $ \R^n.$ We first consider the negative powers $ H^{-s} $ where $ s \geq 0 $ which are given in terms of the Hermite semigroup $ e^{-tH} $ via the Gamma integral:
$$ H^{-s}f(x) = \frac{1}{\Gamma(s)} \int_0^\infty e^{-tH}f(x)  t^{s-1} dt.$$  The kernel of the semigroup $ e^{-tH} $ is explicitly known and is given in terms of the Mehler's formula for the Hermite functions, see \cite{ST}. However, we can also write $ e^{-tH} $ as the Weyl transform of a function on $ \C^n $ which allows us to realise $ e^{-tH} $ and hence $ H^{-s} $ as a pseudo-differential operator. Recall that the Weyl transform $ W(F) $ of a function $ F $ on $ \C^n $ is defined by
$$ W(F) \varphi  = \int_{\C^n} F(z) \pi(z)\varphi dz $$ for $ \varphi \in L^2(\R^n).$ Here, $ \pi(z) $ is the projective representation of $ \C^n $ which is closely related to the Schr\"odinger representations of the Heisenberg group. It is given explicitly by
$$ \pi(x+iy) \varphi(\xi) = e^{i(x \cdot \xi+ \frac{1}{2} x\cdot y)} \varphi(\xi+y).$$ It turns out that $ W(F) $ is an integral operator with kernel
$$ K_F(\xi,\eta) = \int_{\R^n} e^{\frac{i}{2} x \cdot (\xi+\eta)} F(x, \eta-\xi)   dx  $$ where by abuse of notation we have written $ F(x,y) $ in place of $ F(x+iy).$
If $ \tilde{F}(\xi,y) $ stands for the inverse Fourier transform of $ F(x, y) $ in the first set of variables, then we have $ K_F(\xi,\eta) = \tilde{F}(\frac{\xi+\eta}{2},\eta-\xi).$ By letting $ b(\xi,\eta) $ stand for the full inverse Fourier transform of $ F $ in both variables we can write $ W(F) $ as
$$ W(F)\varphi(\xi) = (2\pi)^{-n} \int_{\R^n} \int_{\R^n} e^{i(\xi-\eta)\cdot y}b(\frac{\xi+\eta}{2},y) \varphi(\eta) dy d\eta .$$
Thus we see that the Weyl transform $ W(F) $ is a pseudo-differential operator in the Weyl calculus with symbol $ b(x,\xi).$

We now make use of the well known fact that $ e^{-tH} = W(p_t) $ where 
$$ p_t(z) = c_n (\sinh t)^{-n} e^{-\frac{1}{4} (\coth t) |z|^2} $$ is the heat kernel associated to the so called special Hermite opereator, see e.g. \cite{ST}.  In view of the relation between  a function $ F $ and the Weyl symbol of $ W(F)$, we observe that the Weyl symbol of the Hermite semigroup $ e^{-tH} $ is given by the function $ a_t(x,\xi) = c_n (\cosh t)^{-n} e^{-(\tanh t) (|x|^2+|\xi|^2)}.$ As $ \Gamma(s) H^{-s} = \int_0^\infty t^{s-1} e^{-tH} dt $ the Weyl symbol of $ H^{-s} $ is given by
$$ b_s(x,\xi) = \frac{c_n}{\Gamma(s)}  \int_0^\infty t^{s-1}  (\cosh t)^{-n} e^{-(\tanh t) (|x|^2+|\xi|^2)} dt .$$
By  taking $ s = 1$ and making a change of variables we see that the Weyl symbol $ b(x,\xi) $ of $ H^{-1} $ is given by
$$   b(x,\xi) = c_n \int_0^1 (1-t^2)^{n/2-1} e^{-t(|x|^2+|\xi|^2)} dt.$$
It is an easy matter to prove Theorem 1.1.

{\bf Proof of Theorem 1.1} Let $ b_{2n} $ stands for the Weyl symbol of $ H^{-1} $ on $ \R^{2n} $ given by the above expression. Then expanding $ (1-t^2)^{n-1} $ and making a change of variables we get
$$ b_{2n}(x,\xi) = \sum_{j=0}^{n-1} \frac{(n+j-1)!}{j!(n-1)!} (-1)^j \Big(  \int_0^{(|x|^2|+|\xi|^2)} t^{2j} e^{-t}dt \Big) (|x|^2+|\xi|^2)^{-2j-1}.$$
The proof is completed by showing that 
$  \frac{1}{j!} \int_0^a  t^{j} e^{-t}dt = 1-e^{-a}p_j(a) $ where $ p_j $ are the Taylor polynomials of $ e^{-t}.$ But this follows immediately by induction.

In \cite{CRT} the authors have studied $ H^{-1} $ as a pseudo-differential operator. For the Weyl symbol $ b(x,\xi) $ of $ H^{-1} $ the authors have proved  the estimate
$$  |\partial_{x,\xi}^\alpha b(x,\xi) | \leq C^{|\alpha|+1} (\alpha!)^{(r+1)/2} (|x|^2+|\xi|^2)^{-1-(r/2)|\alpha|} $$ for some constant $ C $ which is independent of $ \alpha \in \mathbb N^{2n} $ and  $ r \in [0,1].$ The proof given in \cite{CRT} is quite long and uses several results from microlocal analysis. Here we give  a very  short proof of the same.

\begin{thm} For $ 0 < s \leq 1$ we have the following estimates on the Weyl symbol $ b_s(x,\xi) $ of the operator $ H^{-s}$: there exists a constant constant $ C>0 $ such that for all  $ \alpha \in \mathbb N^{2n} $ and  $ r \in [0,1]$
$$   |\partial_{x,\xi}^\alpha b_s(x,\xi) | \leq C^{|\alpha|+1} (\alpha!)^{(r+1)/2} (|x|^2+|\xi|^2)^{-s-(r/2)|\alpha|} .$$ 
\end{thm}
\begin{proof}
We make use of some properties of the Hermite functions on $ \R^n.$ Recall that Hermite polynomials  $ H_k(t) $ on the real line are defined by the equation 
$$ H_k(t)  = (-1)^k e^{t^2} \frac{d^k}{dt^k} e^{-t^2} $$
and the normalised Hermite functions  are given by $  h_k(t) = (2^k k! \sqrt{\pi})^{-1/2} H_k(t) e^{-\frac{1}{2}t^2}.$ It is then well known that $ h_k(t) $ are bounded functions uniformly in $ k.$  The multi-dimensional Hermite functions $ H_\alpha(x), x \in \R^n, \alpha \in \mathbb{N}^n $ are defined by taking tensor products. Thus the $ 2n$-dimensional Hermite polynomials $ H_\alpha, \alpha \in \mathbb N^{2n} $ are defined by the equation
$$ H_{\alpha}(x,\xi)  e^{-(|x|^2+|\xi|^2)}= (-1)^{|\alpha|}  \partial_{x,\xi}^\alpha e^{-(|x|^2+|\xi|^2)}.$$ Therefore, from the integral representation for $ b_s $ we obtain  the relation
$$ \partial_{x,\xi}^\alpha b_s(x,\xi) = (-1)^{|\alpha|} \frac{c_n}{\Gamma(s)}  \int_0^\infty t^{s-1}  (\cosh t)^{-n} (\tanh t)^{\frac{1}{2}|\alpha|} H_\alpha((\tanh t)^{1/2}(x,\xi)) e^{-(\tanh t) (|x|^2+|\xi|^2)} dt .$$

We now make use of the fact that the normalised Hermite functions $ \Phi_{\alpha}(x,\xi) $ defined by
$$  \Phi_{\alpha}(x,\xi) = (2^{|\alpha|} (\alpha !)\pi^{n})^{-1/2} H_\alpha(x,\xi) e^{-\frac{1}{2}(|x|^2+|\xi|^2)} $$ are uniformly bounded (which follows from the fact that $ h_k(t) $ are uniformly bounded). This leads to the estimate
$$ |\partial_{x,\xi}^\alpha b_s(x,\xi)| \leq C_n 2^{\frac{1}{2}|\alpha|} (\alpha !)^{1/2}   \int_0^\infty t^{s-1}  (\cosh t)^{-n} (\tanh t)^{\frac{1}{2}|\alpha|}  e^{-\frac{1}{2}(\tanh t) (|x|^2+|\xi|^2)} dt .$$
In order to estimate the integral appearing above, we write it as
$$  I = \int_0^\infty  \Pi_{j=1}^n t^{(s-1)/n}  (\cosh t)^{-1} (\tanh t)^{\frac{1}{2}\alpha_j}  e^{-\frac{1}{2n}(\tanh t) (|x|^2+|\xi|^2)} dt .$$
Applying generalised Holder's inequality, we are led to estimating $ I \leq \Pi_{j=1}^n I_j^{1/n} $ where 
$$  I_j = \int_0^\infty  t^{s-1}  (\cosh t)^{-n} (\tanh t)^{\frac{n}{2}\alpha_j}  e^{-\frac{1}{2}(\tanh t) (|x|^2+|\xi|^2)} dt .$$
Assuming $ s = 1 $ and making a change of variables, we have to estimate the integral
$$ J = \int_0^1 (1-t^2)^{n/2-1} t^{nk/2} e^{-\frac{n}{2}ta^2} dt.$$
Further assuming that $ n \geq 2$ we get two kinds of estimates for $ J.$ Namely, $  J \leq C a^{-2}$  and $ J \leq C \Gamma(1+(nk)/2) a^{-2-nk}.$ These estimates immediately lead to the estimates $ I \leq C  (|x|^2+|\xi|^2)^{-1} $ and
$$  I \leq  C^{|\alpha|} (\alpha !)^{1/2}  (|x|^2+|\xi|^2)^{-1-(1/2)|\alpha|} $$
where we have used Stirling's formula to estimate the Gamma function.  Thus we have proved 
$$ |\partial_{x,\xi}^\alpha b_s(x,\xi)| \leq C^{|\alpha|} (\alpha !)^{1/2} (|x|^2+|\xi|^2)^{-1} $$
as well as 
$$ |\partial_{x,\xi}^\alpha b_s(x,\xi)| \leq C^{|\alpha|} (\alpha !) (|x|^2+|\xi|^2)^{-1-(1/2)|\alpha|} .$$ Interpolation now gives the required estimate when $ s =1 .$ 
When $ 0 < s <1,$ we are led to estimate the integrals
$$  \int_0^\infty t^{s-1}  (\cosh t)^{-n} (\tanh t)^{\frac{1}{2}|\alpha|}  e^{-\frac{1}{2}(\tanh t) (|x|^2+|\xi|^2)} dt .$$ As $ \tanh t $ behaves like $ t $ for $ t $ small and is dominated by $ t $  for $ t\geq 1$ and since $ s-1 <0 $ we can bound the above integral by
$$  \int_0^\infty (\tanh t)^{s-1}  (\cosh t)^{-n} (\tanh t)^{\frac{1}{2}|\alpha|}  e^{-\frac{1}{2}(\tanh t) (|x|^2+|\xi|^2)} dt .$$ This can be estimated as before yielding the required estimate.

\end{proof}

\section{More on fractional powers  of the Hermite operator}

As noted elsewhere, it is sometimes more convenient to use a variant of the fractional power. At least in the case of the sublaplacian $ \mathcal{L} $ on the Heisenberg group $ \H^n $ , it has turned out to be more natural and fruitful to use the conformally invariant fractional power $ \mathcal{L}_s $ instead of the pure fractional power $ \mathcal{L}^s $, see \cite{RT} for the definition. For the case of the Hermite operator it amounts to replace $ H^s $ by the operator defined by 
$$  H_s \varphi = \sum_{k=0}^\infty  \frac{\Gamma(\frac{2k+n+1+s}{2})}{\Gamma(\frac{2k+n+1-s}{2})} P_k\varphi  $$
where $ P_k $ are the spectral projections associated to $ H.$ In view of Stirling's formula for the Gamma function, it follows that $ H_s $ differs from the pure power $ H^s $ by a bounded operator $ U_s.$ Indeed, if we let 
$$ U_s \varphi = \sum_{k=0}^\infty  \frac{\Gamma(\frac{2k+n+1+s}{2})}{\Gamma(\frac{2k+n+1-s}{2})} (2k+n)^{-s} P_k\varphi  $$
then clearly, $ U_s $ is bounded on $ L^2(\R^n) $ and $ H_s = U_s H^s.$ We also note that $ H_s^{-1} = H_{-s}.$ Using the connection between $ \mathcal{L}_s $ and $ H_s $ we can obtain an explicit formula for the Weyl symbol of $ H_s^{-1}.$

We make use of several known facts: first of all we recall (see \cite{ST}) that $ P_k = (2\pi)^{-n} W(\varphi_k) $ where $ \varphi_k(z) = L_k^{n-1}(\frac{1}{2}|z|^2) e^{-\frac{1}{4}|z|^2} $ are the Laguerre functions of type $ (n-1) $ on $ \C^n.$  Here $ L_k^{\alpha}(r) $ are Laguerre polynomials of type $ \alpha.$ Thus if we let
$$ F_s(z) = (2\pi)^{-n} \sum_{k=0}^\infty  \frac{\Gamma(\frac{2k+n+1-s}{2})}{\Gamma(\frac{2k+n+1+s}{2})} \varphi_k(z)  $$
then it follows that $ H_s^{-1}  = W(F_s).$ The function $ F_s $ is known explicitly. To see this, let 
$ \varphi_k^\alpha(r) = L_k^{\alpha}(r^2) e^{-\frac{1}{2} r^2} $ be Laguerre functions of type $ \alpha.$  Let $ K_\nu(r)$ stands for the Macdonald function of type $ \nu$ defined by the Sommerfeld integral (see \cite{NU} p.226)
$$ K_\nu(r) = \frac{1}{2} (\frac{r}{2})^\nu \int_0^\infty e^{-(t+\frac{r^2}{4t})} t^{-\nu-1} dt.$$
Then the function $ G_{\alpha,\sigma}(r) $ defined by
$$  G_{\alpha,\sigma}(r) =  \frac{ 2^{\alpha+\sigma}\Gamma(\frac{\alpha-\sigma}{2})}{\sqrt{\pi} \Gamma(\sigma)}  r^{-\alpha-1+\sigma} K_{(\alpha+1-\sigma)/2}(\frac{1}{2}r^2) $$
can be expanded in terms of the functions $ \varphi_k^\alpha.$ In \cite{CLT} the authors have shown that
$$ G_{\alpha,\sigma}(r) = \frac{2}{\Gamma(\alpha+1)} \sum_{k=0}^\infty  \frac{\Gamma(\frac{2k+\alpha+1+1-\sigma}{2})}{\Gamma(\frac{2k+\alpha+1+1+\sigma}{2})} \varphi_k^\alpha(r).  $$
Thus  we see that, by choosing $ \alpha = n-1 $ and  $ \sigma = s ,$ the function $ F_s $ is explicitly given by
$$ F_s(z) = c_{n,s}  |z|^{-n+s} K_{(n-s)/2}(\frac{1}{4}|z|^2) $$  where $ c_{n,s} $ is an explicit constant. Finally the Weyl symbol of $ H_s^{-1} $ is given by
$$ b_s(x,\xi)  = (2\pi)^{-n}  \int_{\R^{2n}} F_s(u+iv) e^{-i(x \cdot u+\xi \cdot v)} du dv.$$

\begin{thm} For $ 0 < s \leq 1 $ the Weyl symbol of $ H_s^{-1} $ is given explicitly by 
$$ b_s(x,\xi) = c_{n,s} \int_0^1 e^{-a(|x|^2+|\xi|^2)} a^{s-1} (1-a^2)^{\frac{(n-s-1)}{2}} da.$$
Moreover, the following estimates are valid:
$$   |\partial_{x,\xi}^\alpha b_s(x,\xi) | \leq C^{|\alpha|+1} (\alpha!)^{(r+1)/2} (|x|^2+|\xi|^2)^{-s-(r/2)|\alpha|} $$ for some constant $ C $ which is independent of $ \alpha \in \mathbb N^{2n} $ and  $ r \in [0,1].$
\end{thm}
\begin{proof} In order to get the integral representation for $ b_s(x,\xi) $ we make use of the Poisson integral representation of $ K_\nu$: ( see \cite{NU}, p.223)
$$ K_\nu(r) = \frac{\sqrt{\pi}}{\sqrt{2}\Gamma(\nu+1/2)} r^{-1/2} e^{-r} \int_0^\infty e^{-t} t^{\nu-1/2}(1+t/(2r))^{\nu-1/2} dt.$$
Recalling the formula for $ F_s(z) $ in terms of $ K_{(n-s)/2}(\frac{1}{4}|z|^2) $ and using the fact that the Fourier transform of $ e^{-t|z|^2} $ is a constant multiple of $ t^{-n} e^{-\frac{1}{4t}|z|^2} $ we see that, after a change of variables,
$$ b_s(x,\xi) = c_{n,s} \int_0^1  e^{-a(|x|^2+|\xi|^2)} a^{s-1} (1-a^2)^{\frac{(n-s-1)}{2}} da.$$
We observe that the above expression coincides with the formula we got for $ b_1 $ earlier. Estimating  derivatives of $ b_s $ is done as in the case of $ s =1.$ We leave the details to the reader.

\end{proof}

\begin{rem} The integral representation for $b_s $ can also be  obtained easily by making use of the numerical identity
 (see \cite[p. 382, 3.541.1]{GR})
$$
\int_0^{\infty}e^{-\mu t}\sinh^\nu\beta t\,dt=\frac{1}{2^{\nu+1}}
\frac{\Gamma\big(\frac{\mu}{2\beta}-\frac{\nu}{2}\big)\Gamma(\nu+1)}
{\Gamma\big(\frac{\mu}{2\beta}+\frac{\nu}{2}+1\big)},
$$
which is valid for $\operatorname{Re}\beta>0$, $\operatorname{Re}\nu>-1$, $\operatorname{Re}\mu>\operatorname{Re}\beta\nu$. The proof given above has the added advantage that the Fourier transform of $ b_s $ is given explicitly. Indeed, we have
$$ \int_{\R^{2n}} b_s(x,\xi) e^{-i(y \cdot x+ \xi \cdot \eta)} dx d\xi = C_{n,s} (|y|^2+|\eta|^2)^{-(n-s)/2} K_{(n-s)/2}(\frac{1}{4}(|y|^2+|\eta|^2)).$$
Since $ K_\nu $ is a linear combination of  the modified Bessel functions $ I_\nu $ and $ I_{-\nu}, $ (see \cite{NU}, p.224), the above is an explicit formula for the Fourier transform of $ b_s.$

\end{rem}

%%%%%%%%%%%%%%%%%%%%%%%%%%%%%%%%%%%%%%%%%%%%%%%%%%%%%%

\end{document}